\documentclass[12pt,a4paper]{amsart}
\usepackage{amssymb,enumerate}
\usepackage{hyperref}
\usepackage{color}
\usepackage[dvipsnames]{xcolor}
\usepackage[shortlabels]{enumitem}

\newcommand{\CC}{{\mathbb{C}}}
\newcommand{\FF}{{\mathbb{F}}}
\newcommand{\GG}{{\mathbb{G}}}
\newcommand{\LL}{{\mathbb{L}}}
\newcommand{\MM}{{\mathbb{M}}}
\newcommand{\QQ}{{\mathbb{Q}}}
\newcommand{\TT}{{\mathbb{T}}}
\newcommand{\ZZ}{{\mathbb{Z}}}

\newcommand{\bG}{{\mathbf{G}}}
\newcommand{\bL}{{\mathbf{L}}}
\newcommand{\bM}{{\mathbf{M}}}

\newcommand{\bT}{{\mathbf{T}}}
\newcommand{\bu}{{\mathbf{u}}}
\newcommand{\bx}{{\mathbf{x}}}

\newcommand{\cH}{{\mathcal{H}}}

\newcommand{\Deg}{\operatorname{Deg}}
\newcommand{\End}{\operatorname{End}}
\newcommand{\Frac}{\operatorname{Frac}}

\newcommand{\Id}{\operatorname{Id}}
\newcommand{\Irr}{\operatorname{Irr}}
\newcommand{\sgn}{{\operatorname{sgn}}}
\newcommand{\Uch}{{\operatorname{Uch}}}

\newcommand{\GL}{\operatorname{GL}}

\newcommand{\tw}[1]{{}^{#1}\!}
\newcommand{\Chevie}{{\sf Chevie}{}}
\newcommand{\ph}[1]{\phi_{#1}}
\newcommand\co{{:}}
\newcommand\bbe{\boldsymbol{\beta}}

\let\ga=\gamma
\let\la=\lambda
\let\vhi=\varphi
\let\si=\sigma
\let\tht=\theta
\let\ze=\zeta

\theoremstyle{plain}

\newtheorem{thm}{Theorem}[section]
\newtheorem{lem}[thm]{Lemma}
\newtheorem{prop}[thm]{Proposition}
\newtheorem{conj}[thm]{Conjecture}
\newtheorem{cor}[thm]{Corollary}

\theoremstyle{definition}
\newtheorem{rem}[thm]{Remark}

\begin{document}

\title[Intersections of blocks of cyclotomic Hecke algebras]{Intersections of blocks\\ of cyclotomic Hecke algebras}

\author{Maria Chlouveraki}
\address{National and Kapodistrian University of Athens, Department of
  Mathematics, Panepistimioupolis 15784 Athens, Greece.}
\makeatletter
\email{mchlouve@math.uoa.gr}
\makeatother

\author{Gunter Malle}
\address{FB Mathematik, RPTU Kaiserslautern, 67653 Kaiserslautern, Germany.}
\makeatletter
\email{malle@mathematik.uni-kl.de}
\makeatother

\begin{abstract}
Trinh and Xue have proposed a startling conjecture on intersections of blocks
of cyclotomic Hecke algebras occurring in modular representation theory of
finite reductive groups. We prove this conjecture for all exceptional type
groups apart from a few situations in type $E_8$. We also give a conceptual
proof in all cases where relative Weyl groups are cyclic.
Furthermore, we propose several generalisations, to Suzuki and Ree
groups, to non-rational Coxeter groups and even more generally to spetsial
complex reflection groups, and confirm these in various cases.
\end{abstract}



\date{\today}

\thanks{The research of the first author was supported by the Hellenic
 Foundation for Research and Innovation (H.F.R.I.) under the Basic Research
 Financing (Horizontal support for all Sciences), National Recovery and
 Resilience Plan (Greece 2.0), Project Number: 15659, Project Acronym: SYMATRAL.
 The second author gratefully acknowledges support by the DFG -- Project-ID
 286237555 -- TRR 195.}

\maketitle


\section{Introduction}

Cyclotomic Hecke algebras were originally introduced by Brou\'e and Malle
\cite{BM93} to provide an explanation for observations made within the modular
representation theory of finite reductive groups, particularly in non-defining
characteristic. In this paper we compare their specialisations at different
roots of unity, motivated by a startling recent conjecture proposed by Trinh
and Xue \cite[Conj.~4.1]{TX23}.

The broader context for this conjecture, as outlined in the work of Trinh and
Xue, suggests a profound generalisation of level-rank dualities, which were
initially observed in Uglov's work concerning higher-level Fock spaces
\cite{Ug00}. These dualities involve Hecke algebras, specifically those defined
in \cite{BM93}, which are conjectured to describe the endomorphisms of Lusztig
induced modules. In the same context, Brou\'e--Malle--Michel also developed a
generalization of Harish-Chandra theory \cite{BMM99}.  The Trinh--Xue
conjecture (which can be found in Section 2 as Conjecture~\ref{conj:TX})
proposes that:
\begin{itemize}
\item For any generic finite reductive group $G$ and integers $d,e > 0$, the
 intersection of a $d$-Harish-Chandra series and an $e$-Harish-Chandra series
 of $G$ is parametrised by a union of blocks of the Hecke algebra of the
 $d$-cuspidal pair at an $e$-th root of unity, and similarly for the Hecke algebra of the $e$-cuspidal pair at a $d$-th root of unity.
\item These parametrisations match the blocks on the two sides.
\end{itemize}
When two blocks match, the bijection between them should in fact lift to a
derived equivalence between associated blocks of rational double affine Hecke
algebras (DAHAs).

These remarkable structures also find parallels in bimodules formed from the
cohomology of affine Springer fibers, as studied by Oblomkov--Yun \cite{OY16}.
Rational DAHAs themselves, introduced by Cherednik to prove Macdonald's
conjectures, control the decomposition numbers of Hecke algebras associated
with underlying reflection groups. The work of Oblomkov--Yun involves actions
of trigonometric and rational DAHAs on the modified cohomology of homogeneous
affine Springer fibres, where Springer actions intertwine with monodromy
actions of certain braid groups. The Trinh--Xue paper arose from attempts to
find formulas for these (DAHA, braid-group) bimodules, with a conjectured
formula for the rational DAHA module connecting it to representations of finite
groups of Lie type. Expectedly, the braid actions factor through suitable Hecke
algebras, thus suggesting a duality between specific instances of these Hecke
algebras.

While the cyclotomic Hecke algebras were initially conceived to (conjecturally)
explain the block distribution of finite reductive groups at non-defining
primes, the blocks of these algebras, as well as those of the classical generic
Iwahori--Hecke algebra of the Weyl group, are typically smaller than those of
the finite reductive group $G$ itself. This inherent difference makes the
proposed relationship, connecting these smaller blocks to the broader structure
of Harish-Chandra series intersections in $G$, quite mysterious, even if one
accepts that cyclotomic Hecke algebras do control the block theory of $G$. 

The main result of our paper is the following:

\begin{thm}   \label{thm:Lie}
 Conjecture~\ref{conj:TX} holds for all finite reductive groups of
 exceptional type, except possibly for type $E_8$ when $d\in\{3,4,6\}$. For $E_8$ and $d=3,4,6$, we have obtained approximate block decompositions that
are in accordance with Conjecture~\ref{conj:TX}.
\end{thm}

We also confirm analogous results for Suzuki and Ree groups.
Further, we propose a natural generalisation of Conjecture~\ref{conj:TX}
to arbitrary finite Coxeter groups and, more broadly, to spetsial complex
reflection groups and prove:

\begin{thm}   \label{thm:spets}
 The generalised Conjecture~\ref{conj:spets} holds for the non-crystallographic
 Coxeter groups as well as for the spetsial primitive complex reflection groups
 $G_4,G_6,G_8$ and $G_{24}$. For $G_{33}$, we have obtained approximate block
 decompositions that are in accordance with Conjecture~\ref{conj:spets}.
\end{thm}

The paper is structured as follows:
In Section~\ref{sec:conj} we introduce the setting, state the conjecture of
Trinh and Xue, collect some tools and describe our general approach. We then
verify the conjecture for groups of exceptional Lie type in
Section~\ref{sec:Lie},
proving Theorem~\ref{thm:Lie}. In Section~\ref{sec:spets} we formulate an
extension of the conjecture to the spetsial setting and prove Theorem~\ref{thm:spets}.

\section{Blocks in Harish-Chandra series}   \label{sec:conj}

\subsection{Cyclotomic Hecke algebras}   \label{subsec:cyc}
Let $\bG$ be a simple algebraic group over an algebraically closed field of
characteristic~$p$ with a Frobenius endomorphism $F:\bG\to\bG$ defining
an $\FF_q$-structure, and $G:=\bG^F$ the group of $F$-fixed points, a finite
reductive group. While most of what is explained below holds more generally for
$\bG$ just being connected reductive, for simplicity we stick to simple groups.
We write $\Uch(G)$ for the set of unipotent characters of $G$.

Let $e\ge1$ be an integer such that the cyclotomic polynomial
$\Phi_e$ divides the order polynomial of $(\bG,F)$. Recall that the centralisers
of $e$-tori of $\bG$ are called \emph{$e$-split Levi subgroups}. These are
certain $F$-stable Levi subgroups (of not necessarily $F$-stable parabolic
subgroups) of $\bG$ (see e.g.\ \cite[3.5.1]{GM20}). A unipotent character
$\la\in\Uch(G)$ is called \emph{$e$-cuspidal} if $\tw*R_\bL^\bG(\la)=0$ for
all proper $e$-split Levi subgroups of $\bG$. Here, $\tw*R_\bL^\bG$ denotes
Lusztig restriction. A \emph{unipotent $e$-cuspidal pair} of $\bG$ is a
pair $(\bL,\la)$ consisting of an $e$-split Levi subgroup $\bL$ of $\bG$ and an
$e$-cuspidal unipotent character $\la\in\Uch(\bL^F)$. These play a fundamental
role in the $\ell$-modular representation theory of $G$ for primes $\ell$
dividing $\Phi_e(q)$, the $e$th cyclotomic polynomial evaluated at~$q$, in that
they serve to parametrise the unipotent $\ell$-blocks of $G$.
Since in this paper we will solely be interested in unipotent characters, we
will just speak of $e$-cuspidal pairs when we mean unipotent $e$-cuspidal pairs.
\par
Associated to any $e$-cuspidal pair $(\bL,\la)$ is its \emph{relative Weyl
group} defined as $W_\bG(\bL,\la):=N_G(\bL,\la)/\bL^F$. If the ambient group
$\bG$ is clear from the context, we will also denote it by $W(\bL,\la)$. It is
known by inspection that $W(\bL,\la)$ is always a finite complex reflection
group over an abelian number field $K$, irreducible if $\bG$ is simple (see
\cite[\S3]{BM93} and \cite[Tab.~3]{BMM}).
Let $\cH(W(\bL,\la),\bu)$ be the generic Hecke algebra of $W(\bL,\la)$ over
$\ZZ[\bu^{\pm1}]$, for $\bu$ a set of algebraically independent parameters
(as defined for example in \cite[1.5]{BMM14}). In \cite{BM93} and \cite{Ma97}
we single out a particular 1-parameter specialisation $\cH(W(\bL,\la),x)$ of
$\cH(W(\bL,\la),\bu)$ with respect to a specialisation map
$\ZZ[\bu^{\pm1}]\to K[x^{\pm\frac{1}{2}}]$ depending on $\bG$ and $(\bL,\la)$,
the \emph{cyclotomic Hecke algebra} associated to $(\bL,\la)$ in $\bG$.
Conjecturally, the specialisation of $\cH(W(\bL,\la),x)$ at $x\mapsto q$ should
be the endomorphism algebra of the complex defining Lusztig induction
$R_\bL^\bG(\la)$, see \cite[($d$-HV3)]{BM93}; this has been shown in a number
of situations by Lusztig, Digne, Michel, Rouquier and Dudas (see
\cite[Thm~3.6]{TX23} for a list of known instances), but is open in general.
Nevertheless, all numerical consequences implied by this have been shown to
hold \cite{Ma00}.

Now the Hecke algebra $\cH(W(\bL,\la),x)$ (conjecturally) carries a canonical
symmetrising form, which should, in particular, specialise to the natural form
on the endomorphism algebra of $R_\bL^\bG(\la)$. Existence of this form is known
when $W(\bL,\la)$ is a Weyl group and has been established case-by-case for all
but a handful of primitive complex reflection groups: for the imprimitive groups in \cite{BreMa, MM98}, and for
\begin{itemize}
    \item the group $G_4$ in \cite{MM10, MW, BCCK} (3 independent proofs); \smallbreak
	\item the groups $G_5, G_6, G_7, G_8$ in \cite{BCCK}; \smallbreak
	\item the group $G_{12}$ in \cite{MM10};  \smallbreak
	\item the group $G_{13}$ in \cite{BCC}; \smallbreak
    \item the groups $G_9, G_{10}, G_{11}, G_{14}, G_{15}$ in \cite{CP25} (where all groups $G_4$--$G_{15}$ are treated); \smallbreak
	\item the groups $G_{22}, G_{24}$ in \cite{MM10}.
\end{itemize}	

Let $\Irr(\cH(W(\bL,\la),x))$ denote the
set of characters of the irreducible representations of
$\cH(W(\bL,\la),x))$ over a splitting field. For
$\chi\in\Irr(\cH(W(\bL,\la),x))$ we denote its Schur elements with respect to
this form by $S_\chi$, an element of $K[x^{\pm\frac{1}{m}}]$ for some suitable
$m\ge1$. These Schur elements have been computed, assuming the form exists,
in all cases (see \cite{Ma00} and the references therein).

Let $\Uch(G,(\bL,\la)):=
\{\chi\in\Uch(G)\mid \langle R_\bL^\bG(\la),\chi\rangle\ne0\}$ denote the
\emph{$e$-Harish-Chandra series of $G$ above} the $e$-cuspidal pair $(\bL,\la)$
and recall that by the results of Lusztig, attached to any unipotent character
$\chi$ is a degree polynomial $\Deg(\chi)\in\QQ[x]$ such that
$\Deg(\chi)(1)=\chi(1)$, see \cite[Def.~2.3.25]{GM20}. We then have:

\begin{thm} \label{thm:label}
 For each $e$-cuspidal pair $(\bL,\la)$ there is a bijection
 $$\Psi_{\bL,\la}:\Irr(\cH(W(\bL,\la),x))\longrightarrow\Uch(G,(\bL,\la))$$
 satisfying
 $$\Deg(\Psi_{\bL,\la}(\chi))=\pm\Deg(\la)\,|\bG:\bL|_{x'}/\,S_\chi\qquad
   \text{for all $\ \chi\in\Irr(\cH(W(\bL,\la),x))$}.\eqno{(*)}$$
\end{thm}

Here, $|\bG:\bL|_{x'}$ denotes the part prime to $x$ of the order polynomial
of $(\bG,F)$ divided by the one of $(\bL,F)$.

The first part of the above theorem is \cite[Thm~3.2]{BMM}.  The second part was conjectured in \cite[($d$-HV6)]{BM93} and the proof 
was completed in \cite[Prop. 7.1]{Ma00}, using earlier work of several authors
on Schur elements (as explained in the paragraph before \cite[Prop. 7.1]{Ma00}).

A collection of bijections
$\Psi:=\{\Psi_{\bL,\la}\mid e\ge1,\ (\bL,\la)\text{ $e$-cuspidal}\}$ as in
Theorem~\ref{thm:label} will be called a \emph{complete HC-parametrisation}.

\begin{rem}
 For $e=1$ the bijections $\Psi_{\bL,\la}$ come from the identification of the
 cyclotomic Hecke
 algebras $\cH(W(\bL,\la),q)$ as endomorphism algebras of Harish-Chandra induced
 modules and provide the standard labelling of unipotent characters of $G$.
 For $e>1$ the bijections should be induced by the cited conjectural relation to
 endomorphism algebras as well; this has not been proved, and condition~$(*)$
 does not, in general, determine $\Psi_{\bL,\la}$ uniquely though.
\end{rem}

\subsection{The conjecture of Trinh and Xue}
We can now describe the conjecture of Trinh and Xue. For this,
consider a second integer $d\ge1$ such that $\Phi_d$ also divides the order
polynomial of $(\bG,F)$, and let $(\bM,\mu)$ be a $d$-cuspidal pair of $\bG$,
with associated cyclotomic Hecke algebra $\cH(W(\bM,\mu),x)$. Let $\ze_d,\ze_e$
be, respectively, primitive $d$th and $e$th roots of unity, and consider the
specialisations $\cH(W(\bL,\la),\ze_d)$ and $\cH(W(\bM,\mu),\ze_e)$. These
will, in general, not be semisimple, so will induce some non-trivial block
partitions on $\Irr(\cH(W(\bL,\la),x))$ and $\Irr(\cH(W(\bM,\mu),x))$
respectively.
Note that when $e=1$, the blocks of $\cH(W(\bL,\la),\ze_d)$ are related to
the $\ell$-block structure of $G$ for primes $\ell$ dividing $\Phi_d(q)$ (see
e.g.\ \cite[\S4]{GJ11}), but no such relation is known (to us) when $e>1$.

Trinh and Xue want to relate these block partitions by comparing their images
under $\Psi$ in the corresponding Harish-Chandra series of $G$. Note that these
will depend on the choice of complete HC-parametrisation $\Psi$.
Concretely, \cite[Conj.~4.1]{TX23} predicts the following (denoted properties~(I)
and~(II) in loc.\ cit.):

\begin{conj}[Trinh--Xue]   \label{conj:TX}
 There exists a choice of complete HC-parametrisation $\Psi=\{\Psi_{\bL,\la}\}$
 such that for any $e$-cuspidal pair $(\bL,\la)$ and $d$-cuspidal pair
 $(\bM,\mu)$ the intersection of Harish-Chandra series
 $$\Psi_{\bL,\la}(\Irr(\cH(W(\bL,\la),x)))
     \ \cap\ \Psi_{\bM,\mu}(\Irr(\cH(W(\bM,\mu),x)))$$
 admits a partition $\bigsqcup B_i$ such that for all $i$,
 $\Psi_{\bL,\la}^{-1}(B_i)$ is a block of $\cH(W(\bL,\la),\ze_d)$
 and $\Psi_{\bM,\mu}^{-1}(B_i)$ is a block of $\cH(W(\bM,\mu),\ze_e)$.
 In particular, this defines a bijection between those blocks of the
 two algebras appearing in this intersection.
\end{conj}

This claim has been verified in \cite[Thm~3]{TX23} for groups of type $\GL_n$
when $d,e$ are coprime, and in \cite[Cor.~5 and \S8]{TX23} certain numerical
consequences are checked for exceptional type groups. Of course the conjecture
makes sense in a more general setting where $\bG$ is just assumed to be
connected reductive, but it easily reduces to the case of simple $\bG$ as
unipotent characters and Lusztig induction behave well with respect to central
products. As we will see below, the conjecture also makes sense and holds when
$F$ is not a Frobenius but a Steinberg map, so $G$ is a Suzuki or Ree group.

\subsection{Verifying the conjecture}   \label{subsec:veri}
We start out by making some general remarks.
The following observation is valid for any $\bG$:

\begin{lem}   \label{lem:def0}
 Assume that $(\bM,\mu)=(\bG,\mu)$ is a $d$-cuspidal pair of $\bG$. Then
 Conjecture~\ref{conj:TX} holds for any $e$-cuspidal pair $(\bL,\la)$,
 independently of the chosen HC-parametrisation.
\end{lem}

\begin{proof}
If $(\bM,\mu)=(\bG,\mu)$, that is, if $\mu$ is a $d$-cuspidal unipotent
character of $G$, then $W(\bG,\mu)=1$, so $\Irr(\cH(W(\bM,\mu),x)=\{1\}$ is a
singleton and hence consists of a single block. It follows by
\cite[Prop.~2.4]{BMM} that $\Deg(\mu)$ is divisible by the full $\Phi_d$-part
of the order polynomial of $(\bG,F)$. Let $(\bL,\la)$ be
an $e$-cuspidal pair of $\bG$ with $\chi\in\Irr(\cH(W(\bL,\la),x))$ such that
$\Psi_{\bL,\la}(\chi)=\mu$. Then by Theorem~\ref{thm:label},
$$\Deg(\mu)=\Deg(\Psi_{\bL,\la}(\chi))=\pm\Deg(\la)|\bG:\bL|_{x'}/S_\chi$$
is divisible by the full $\Phi_d$-part of the order polynomial of $(\bG,F)$.
Now $\Deg(\la)$ divides the order polynomial of $(\bL,F)$, so the Schur element
$S_\chi$ must be prime to $\Phi_d$. That is, its specialisation $S_\chi(\ze_d)$
at $x=\ze_d$ of $S_\chi$ does not vanish. But then by general properties of
symmetric algebras \cite[Thm~7.5.11]{GP}, $\chi$ lies in a singleton block of
$\cH(W(\bL,\la),\ze_d)$, so Conjecture~\ref{conj:TX} holds for the pair
$(\bL,\la),(\bM,\mu)$.
\end{proof}

We also have a general result at the opposite extreme from cuspidal characters.
We know that $\chi,\psi\in\Irr(\cH)$ lie in the same block if and only if the
values of their central characters $\omega_\chi,\omega_\psi$ agree on all
central elements of $\cH$ \cite[Lemma 7.5.10]{GP}.  In particular, if
$\chi,\psi\in\Irr(\cH)$ lie in the same block, then the values of their central
characters $\omega_\chi,\omega_\psi$ agree on a generator of the centre of the
braid group. This defines a (possibly coarser) partition of $\Irr(\cH)$
that we call the \emph{coarse blocks}. In \cite{DM14}, Digne and Michel
conjecture a more precise formulation of the original conjecture
\cite[($d$-HV3)]{BM93} on the structure of endomorphism algebras of the
cohomology of Deligne--Lusztig varieties.

\begin{prop}   \label{prop:reg}
 Assume that $(\bL,1)$, $(\bM,1)$ are respectively $e$-cuspidal and $d$-cuspidal pairs of $\bG$.
 Then (assuming \cite[Conj.~9.1]{DM14} holds) Conjecture~\ref{conj:TX} for
 $(\bL,1)$, $(\bM,1)$ is compatible with the coarse blocks.
\end{prop}

\begin{proof}
Let $\delta\ge1$ be minimal such that $F^\delta$ acts trivially on the Weyl
group of $\bG$.
By \cite[Cor.~9.5]{DM14}, assuming \cite[Conj.~9.1]{DM14}, for any $\ga$ in
the $e$-Harish-Chandra series of $\bG^F$ above $(\bL,1)$ the Frobenius
$F^\delta$ acts by a scalar
$$c_{\bL,1}\la_\ga q^{-\delta(A_\ga+a_\ga)/e}$$
on the $\ga$-isotypic part of a suitable Deligne--Lusztig variety attached
to $(\bL,1)$, where $A_\ga$ is the degree in $x$ of the degree
polynomial of~$\ga$, $a_\ga$ is the power of $x$ dividing it, and $\la_\ga$ is
the Frobenius eigenvalue of $\ga$ and $c_{\bL,1}$ is a constant only depending
on $(\bL,1)$. Thus the corresponding central element of the cyclotomic Hecke
algebra $\cH(W(\bL,1),x)$ acts by
$$c_{\bL,1}\la_\ga x^{-\delta(A_\ga+a_\ga)/e}$$
in the irreducible representation labelled by $\Psi_{\bL,1}(\ga)$, and hence
by the scalar
$$c_{\bL,1}\la_\ga \exp\big(-2\pi i\delta(A_\ga+a_\ga)/(de)\big)$$
in the specialisation $\cH_e:=\cH(W(\bL,1),\ze_d)$. Comparing with the
corresponding scalars for the characters in the $d$-Harish-Chandra series
above $(\bM,1)$ we see that two characters of $\cH_e$ lie in the same coarse
block if and only if their images under $\Psi_{\bM,1}^{-1}\circ\Psi_{\bL,1}$
lie in the same coarse block for~$\cH_d:=\cH(W(\bM,1),\ze_e)$.
\end{proof}

This applies in particular if $d$ is a \emph{regular number} for $W$
(see e.g.\ \cite[Def.~3.1]{BM97}), and $\bL,\bM$ are the centralisers of respectively
Sylow $e$-tori and Sylow  $d$-tori of $\bG$.

\subsection{The cyclic case}   \label{subsec:cyclic}
Assume $W(\bL,\la)$ is cyclic, of order~$m$. Then
$\cH(W(\bL,\la),x)$ is generated by a single element $T$ satisfying a relation
$\prod_{i=1}^m(T-\ga_ix^{a_i})=0$ for suitable roots of unity~$\ga_i$
and half-integral exponents $a_i$ (see \cite[\S2]{BM93}). Thus, all irreducible
characters of $\cH(W(\bL,\la),x)$ (over a splitting field) are linear, and
there are exactly $m$ of them, defined by
$$\chi_i:T\mapsto\ga_ix^{a_i},\qquad i=1,\ldots,m.\eqno{(\dagger)}$$
In the specialisation $\cH(W(\bL,\la),\ze_d)$, two characters $\chi_i,\chi_j$
lie in the same block if and only if their specialisations agree, so if and
only if
$\ga_i\ze_d^{a_i}=\ga_j\ze_d^{a_j}$. Hence in this case the determination
of blocks is immediate from the parameters defining the cyclotomic Hecke
algebra.

\begin{cor}   \label{cor:cyc}
 Assume $d,e$ are both regular.
 Let $(\bL,1)$, $(\bM,1)$ be respectively $e$-cuspidal and  $d$-cuspidal pairs such that
 both relative Weyl groups $W(\bL,1)$ and $W(\bM,1)$ are cyclic. Then
 (assuming \cite[($d$-HV3)]{BM93} holds) Conjecture~\ref{conj:TX} holds for the
 intersection of the corresponding two Harish-Chandra series.
\end{cor}

\begin{proof}
By the remark after \cite[Lemma~5.11]{BM97} in this case $\cH$ is generated
by the element $F$ and hence by what we observed just above, coarse blocks
agree with blocks, so the claim follows by Proposition~\ref{prop:reg}.
\end{proof}

A general formula for values of characters of cyclotomic Hecke algebras on
central elements is given in \cite[Prop.~6.15]{BMM99}.

\subsection{Ennola duality}   \label{subsec:Ennola}
There is a further simplification coming from \emph{Ennola
duality} (see \cite[(1.5), (1.21), Thm~3.2]{BMM}). For integers $e\ge1$ let
$$e':=\begin{cases} 2e& \text{if $e\equiv1\pmod2$},\\
        e/2& \text{if $e\equiv2\pmod4$},\\
        e& \text{if $e\equiv0\pmod4$}.\end{cases}$$
Then we have $\Phi_e(-x)=\pm\Phi_{e'}(x)$ (where the minus sign occurs only when
$e\in\{1,2\}$). Now associated to $(\bG,F)$ is its Ennola dual group $(\bG,F')$
whose complete root datum is obtained by formally replacing $F$ by~$-F$.
We have the following:

\begin{lem}   \label{lem:Ennola}
 Let $\bG$ be as above and $d,e\ge1$.
 \begin{enumerate}[\rm(a)]
  \item If Conjecture~\ref{conj:TX} holds for $(\bG,F)$ at $(d,e)$ it also holds
   for $(\bG,F')$ at $(d',e')$.
  \item If $W$ contains $-\Id$ then Conjecture~\ref{conj:TX} holds for $(e,e')$
   for all $e\ge1$ with respect to a suitable complete HC-parametrisation.
 \end{enumerate}
\end{lem}

\begin{proof}
The unipotent characters of the finite reductive group $\bG^{F'}$ are in
bijection with those of $\bG^F$, where the degree functions are changed by
replacing $x$ by $-x$ \cite[Thm~3.3]{BMM}. Further, for any $F$-stable maximal
torus $\bT$ of $\bG$
there exists an $F'$-stable maximal torus $\bT'$ of $\bG$ whose order
polynomial is obtained from the one of $\bT$ by replacing $x$ by $-x$. Also,
the relative Weyl groups of $\bT$ and $\bT'$ agree and thus so do their generic
Hecke algebras. More generally, the $e$-cuspidal pairs of $(\bG,F)$ are in
bijection with the $e'$-cuspidal pairs of $(\bG,F')$, with equal relative Weyl
groups and thus Hecke algebras. It is part of the conjectures in \cite{BM93}
that the cyclotomic specialisations for cuspidal pairs of $(\bG,F')$ are
obtained from the ones for $(\bG,F)$ by replacing $x$ by $-x$. Then clearly the
$\ze_d$-blocks of the former are mapped to $\ze_{d'}$-blocks of the latter. So
if Conjecture~\ref{conj:TX} holds for $(\bG,F)$ and $(d,e)$, it will
automatically hold for $(\bG,F')$ and $(d',e')$.
\par
If $W$ contains $-\Id$ then $(\bG,F)$ and $(\bG,F')$ are isomorphic and hence
their sets of unipotent characters, cuspidal pairs and relative Weyl groups
may be identified. Let $(\bL,\la)$ be an $e$-cuspidal pair of $(\bG,F)$ with
cyclotomic Hecke algebra $\cH:=\cH(W(\bL,\la),x)$. Then clearly, the
specialisation of $\cH$ at $x$ a primitive $e'$th root of unity is identical to
the specialisation of $\cH$ at $-x$ a suitable primitive $e$th root of unity,
by the definition of $e'$. Hence, the blocks of both are the same and choosing
a suitable HC-parametrisation the claim follows.
\end{proof}

\subsection{Algorithmic determination of blocks}   \label{subsec:algo}
We now describe how we determined blocks of specialised Hecke algebras. Let
$W$ be a complex reflection group and let $\cH(W)$ denote a cyclotomic
Hecke algebra over $R:=\CC[x,x^{-1}]$, where $y^z=x$ with $z=|Z(W)|$.
It is known that $\Frac(R)$ is a splitting field for $\cH(W)$. A
representation $(V,\rho)$ of $\cH(W)$ is a $\cH(W)$-module $V$ which is
finite-dimensional free as $R$-module, with structural morphism
$\rho:\cH(W)\to\End(V)$.
Let $\ze_e\in\CC$ be a primitive $e$th root of unity and let $\ze\in\CC$ with
$\ze^z=\ze_e$. We consider the ring homomorphism $\tht:R\to\CC$ defined by
$\tht(y)=\ze$, so $\tht(x)=\ze_e$, and let $\cH_e(W)$ denote the corresponding
specialised Hecke algebra over $\CC$ (note that this depends on the choice of
$\ze$). If $\chi\in\Irr(W)$, then we denote by $S_\chi$ its Schur element in
$\cH(W)$ and we call defect of $\chi$ with respect to $\tht$ the multiplicity
of $\ze$ as a root of $S_\chi$.

If $W$ is a cyclic group, then the blocks of $\cH_e(W)$ are easy to compute,
because two characters $\chi,\psi$ lie in the same block if and only if
$\tht(\chi(h))=\tht(\psi(h))$ for all $h \in \cH(W)$, see
Section~\ref{subsec:cyclic} above.

The blocks of $\cH_e(W)$ for the exceptional finite Coxeter groups $W$ can be
found in
\cite[App.~F]{GP}, while for many cases of groups of rank 2 they have
been calculated in \cite{CM11}. The algorithm used in the latter reference is
similar to the one we applied to the rest of the cases. This algorithm is
heuristic and it demands that we apply all of the following arguments to
determine mostly when two characters do not lie in the same block. In most
instances, we explicitly computed the decomposition matrix
$D=(d_{\chi,\vhi})_{\chi\in\Irr(W),\vhi\in\Irr(\cH_e(W))}$. For $G_{25}$ a
couple of cases were taken from \cite{Pl14}.

\begin{enumerate}[(1)]
\item A character $\chi$ is alone in its block if and only if
 $\tht(S_\chi) \neq 0$ (\cite[Thm~7.5.11]{GP} and \cite[Prop.~4.4]{GR}).
\item Any representation of dimension $1$ is irreducible. A representation of
 dimension~$2$ is irreducible, unless it has a 1-dimensional subrepresentation.
 A representation of dimension~$3$ is irreducible unless it or its transpose
 has a 1-dimensional subrepresentation. Dimension $1$ subreprepresentations
 correspond to common eigenvectors of the matrices representing the generators
 of $\cH(W)$, so they are easy to find.
\item In general, assume that $(V,\rho)$, $(V',\rho')$ are simple modules of
 $\cH(W)$. Then $(V,\tht\circ\rho)$, $(V',\tht\circ\rho')$ are modules of the
 specialised algebra $\cH_e(W)$ (where by $\tht\circ\rho$ we mean that we apply
 $\tht$ to the entries of the matrix $\rho(h)$ for all $h\in\cH(W)$).
 If $(V',\tht\circ\rho')$ is isomorphic to a submodule of $(V,\tht\circ\rho)$,
 then there exists a (rectangular) matrix $M$ such that
 $(\tht\circ\rho'(s)) \cdot M = M \cdot (\tht\circ\rho(s))$ for every generator
 $s$ of $\cH(W)$ \cite{Pl14}. We only used this argument for the groups
 $G_{25}$ and $G_{26}$.
\item We have that $\chi,\psi$ are in the same block if and only if
 $\tht(\omega_\chi(z))=\tht(\omega_\psi(z))$ for all $z\in Z(\cH(W))$, where
 $\omega_\chi$ denotes the central character of $\chi$ \cite[Lemma 7.5.10]{GP}.
 A known element of $ Z(\cH(W))$ is the generator $\bbe$ of the centre of the
 corresponding braid group (proved to be cyclic by Bessis
 \cite[Thm~12.8]{Bes15}). So if
 $\tht(\omega_\chi(\bbe))\neq \tht(\omega_\psi(\bbe))$, then $\chi,\psi$ are
 not in the same block.
\item Let $\vhi\in\Irr(\cH_e(W))$ and let $P(x)$ be a polynomial that is
 divisible by $S_\chi$ for all $\chi$ such that $d_{\chi,\vhi}\neq 0$. Then
 (see \cite[Lemma 6]{CM11})
 $$\sum_{\chi \in \Irr(W)} d_{\chi,\vhi} \cdot \tht({P(x)}/{S_\chi})=0.$$
 The fraction ${P(x)}/{S_\chi}$ could simply be the generic degree $D_\chi$,
 but when a character is not of maximal defect in $\cH(W)$ (that is,
 $\tht(D_\chi)=0$), it is convenient to use the least common multiple of the
 Schur elements of the characters in it.
\end{enumerate}

Using the above arguments we were able to explicitly compute the blocks in all
cases except for the groups $G_{31}, G_{32}$ (which occur for $d=4$,
respectively $d=3,6$ in $E_8$) and for $G_{33}$. In these three cases,
using the above algorithm, we were able to split the characters into
``approximate blocks'', but these could still be unions of blocks
(our algorithm determining when characters are not in the same block, but not
when they are).  Note that there is also data missing regarding the
representations of the corresponding Hecke algebras, which made our job more
difficult, especially for $G_{31}$.

\begin{rem}
After numerous calculations of blocks for cyclotomic Hecke algebras, it is
worth pointing out the following facts:
\begin{enumerate}[(1)]
\item We verified that the ``defect conjecture'' \cite[Conj.~2]{CJ23} holds in
  every case, that is, all characters in a block have the same defect. 
\item The irreducible representations of $\cH_e(W)$ are liftable for all
  rank $2$ groups (this is what is called the existence of an optimal basic set
  in \cite{CM11}), but not for rank $3$ or higher.
\item In most of the cases, the converse of the last statement of (4) above
  holds, that is, if $\tht(\omega_\chi(\bbe))= \tht(\omega_\psi(\bbe))$, then
  $\chi,\psi$ are in the same block. However, we have found a few exceptions
  (even in rank $2$ groups), so this does not hold in general. (Examples are
  given by $\cH(G_5)$ with parameters $1,x,x^2$ for $x=\zeta_4$, and
  $\cH(G_{22})$ with parameters $1,x^2$ for $ x=\zeta_5$.)
\end{enumerate}
\end{rem}
All blocks we computed can be found in the github repository \cite{git}.

\section{Blocks in the exceptional types}   \label{sec:Lie}

We now discuss Conjecture~\ref{conj:TX} for the various types of exceptional
groups, by increasing rank. Our notation for unipotent characters follows
\cite{Ca} which is also the one used in the computer algebra package \Chevie
\cite{Mi15} from which all data are taken.

Note that we do not (need to) consider groups of type $\tw2E_6$, since these
are obtained from the untwisted groups $E_6$ by Ennola duality (see
Lemma~\ref{lem:Ennola} above).
Since Conjecture~\ref{conj:TX} is symmetric in $d$ and $e$ and moreover
vacuously satisfied when $d=e$, we may and will assume from now on that $d< e$,
and also that both $\bL,\bM$ are proper Levi subgroups of $\bG$
(by Lemma~\ref{lem:def0}).
\smallskip

It is our understanding that in \cite[\S8]{TX23} it is shown for types
$G_2$ and $F_4$ that the block sizes of the specialised cyclotomic algebras do
match up as predicted by Conjecture~\ref{conj:TX}, but without making explicit
a choice of HC-parametrisation and (hence) also without checking that the images
of the blocks under $\Psi_{\bL,\la},\Psi_{\bM,\mu}$ do actually agree. We check
this latter property using the data and the functionality of the
\Chevie-system \cite{Mi15}.

\subsection{Type $G_2$}   \label{subsec:G2}
We start with type $G_2$.
The block sizes for the relevant algebras have already been determined in
\cite[8.6]{TX23}. Here we also consider the intersection of the corresponding
sets of unipotent characters. There are two pairs of  characters of
$G_2(q)$ not distinguished by their degrees, namely the principal series
characters $\ph{1,3}',\ph{1,3}''$ and the cuspidal characters
$G_2[\tht],G_2[\tht^2]$ of $G_2(q)$, where as customary $\tht$ denotes a
primitive third root of unity. For these, there exist several possible
Harish-Chandra parametrisations. Now it turns out that $\ph{1,3}',\ph{1,3}''$ in
all cases either
lie in the same block or they both occur in a singleton, so for the verification
of Conjecture~\ref{conj:TX} no ambiguity arises. On the other hand, for
$G_2[\tht],G_2[\tht^2]$ a parametrisation via endomorpism algebras is fixed
by \cite[(7.3)]{Lu76} for $d=6$. Hence the only relevant choice is for $d=3$.
We present a compatible HC-parametrisation in Table~\ref{tab:G2}. Here the
characters of $\cH(W(G_2),x)$ are identified through their standard \Chevie-names, the (linear) character for cyclotomic Hecke algebras for cyclic
groups by the corresponding parameter as in~$(\dagger)$. 
With this, we have:

\begin{prop}   \label{prop:G2}
 Conjecture~\ref{conj:TX} holds for type $G_2$ with respect to some complete
 HC-parametrisation. The block intersections are displayed in
 Table~\ref{tab:G2} (up to Ennola duality).
\end{prop}

\begin{table}[htb]
\caption{Block intersections for $G_2$}   \label{tab:G2}
$\begin{array}{c|cc}
 (d,e)& \text{intersections}\\
\hline
 (1,2)& \{\ph{1,0},\ph{1,3}',\ph{1,3}'',\ph{1,6}\}\\
 (1,3)& \{\ph{1,0},\ph{1,6},\ph{2,2}\}\\
 (1,6)& \{\ph{1,0},\ph{1,6},\ph{2,1}\}\\
 (3,6)& \{\ph{1,0},G_2[\tht]\},\{\ph{1,6},G_2[\tht^2]\}\\
\end{array}$
\smallskip

Here, we have omitted all singleton blocks.
\end{table}

\begin{proof}
The relative Weyl groups are cyclic or of type $G_2$. The parameters for the
cyclotomic Hecke algebra are known by the work of Lusztig or predicted in
\cite[Bem.~2.7]{BM93} (see also the table compiled in \cite[8.6]{TX23}). The
decomposition matrices
for specialisations of $\cH(W(G_2),x)$ are given in \cite[Tab.~7.3]{GJ11}, the
ones for the cyclic case were described in Section~\ref{subsec:veri}. The claim follows by direct
computation.
\end{proof}

\subsection{Type $\tw3D_4$}   \label{subsec:3D4}
All unipotent characters of the triality groups $\tw3D_4(q)$ are uniquely
determined by their degree polynomial and hence $(*)$ determines a unique
complete HC-parametrisation.

\begin{prop}   \label{prop:3D4}
 Conjecture~\ref{conj:TX} holds for type $\tw3D_4$; the block intersections are
 displayed in Table~\ref{tab:3D4} (up to Ennola duality).
\end{prop}

\begin{table}[htb]
\caption{Block intersections for $\tw3D_4$}   \label{tab:3D4}
$\begin{array}{c|cc}
 (d,e)& \text{intersections}\\
\hline
  (1,2)&  \{\ph{1,0},\ph{1,3}',\ph{1,3}'',\ph{1,6}\}\\
  (1,3)&  \{\ph{1,0},\ph{1,6},\ph{2,2}\},\{\ph{1,3}',\ph{1,3}'',\ph{2,1}\}\\
  (1,6)&  \{\ph{1,0},\ph{1,3}',\ph{1,3}'',\ph{1,6},\ph{2,2}\}\\
 (1,12)&  \{\ph{1,0},\ph{1,6},\ph{2,1}\}\\
  (3,6)&  \{\ph{1,0},\ph{1,3}',\ph{1,3}'',\ph{1,6},\ph{2,2},\tw3D_4[1]\}\\
 (3,12)& \{\ph{1,0},\ph{1,6},\ph{2,1}\} \\
\end{array}$
\smallskip

Here, again, we have omitted all singleton blocks.
\end{table}

\begin{proof}
The relative Weyl groups are described in \cite[Tab.~3]{BMM}, the parameters
for the corresponding cyclotomic Hecke algebras are predicted in
\cite[Folg.~5.6]{BM93} (see also the table compiled in \cite[8.8.1]{TX23}).
The blocks for non-cyclic $W$ are again to be found in \cite[Tab.~7.4]{GJ11}
for type $W(G_2)$, and have been worked out in \cite[3.3.1]{CM11} for the
primitive complex reflection group~$G_4$, respectively.
\end{proof}

\subsection{Type $F_4$}   \label{subsec:F4}
We now turn to type $F_4$. Again, the block sizes for the relevant algebras have
already been determined by Trinh and Xue \cite[8.7]{TX23}. Here, we also
compare the actual intersections of Harish-Chandra series. This group is
particularly interesting as there are ten pairs of unipotent characters of
$F_4(q)$ having equal degree, so for all of those are there potential choices
in the HC-parametrisation.

\begin{prop}   \label{prop:F4}
 Conjecture~\ref{conj:TX} holds for type $F_4$ with respect to some complete
 HC-parametrisation. The block intersections are displayed in Table~\ref{tab:F4}.
\end{prop}

\begin{table}[htb]
\caption{Block intersections for $F_4$}   \label{tab:F4}
$$\begin{array}{c|cc}
 (d,e)& \text{intersections}\\
\hline
 (1,2)&  \{\ph{1,0},\ph{1,12}^*,\ph{1,24},\ph{2,4}^*,\ph{2,16}^*,\ph{9,2},\ph{9,6}^*,\ph{9,10},\ph{6,6}^*,\ph{8,3}^*,\ph{8,9}^*\},\\
  & \{\ph{4,1},\ph{4,7}^*,\ph{4,13}\},\{B_2[.1^2],B_2[.2],B_2[1^2.],B_2[2.]\}\\
 (1,3)&  \{\ph{1,0},\ph{1,12}^*,\ph{1,24},\ph{2,4}^*,\ph{2,16}^*,\ph{4,8}\}, \{\ph{4,1},\ph{4,7}^*,\ph{4,13},\ph{8,3}^*,\ph{8,9}^*,\ph{16,5}\}\\
 (1,4)&  \{\ph{1,0},\ph{1,24},\ph{4,8},\ph{9,2},\ph{9,10},\ph{6,6}',\ph{12,4},\ph{4,1},\ph{4,13}\},\\
  & \{\ph{2,4}',\ph{2,16}',\ph{4,7}'\},\{\ph{2,4}'',\ph{2,16}'',\ph{4,7}''\},\{B_2[.1^2],B_2[1.1],B_2[2.]\}\\
 (1,6)& \{\ph{1,0},\ph{1,24},\ph{2,4}^*,\ph{2,16}^*,\ph{9,6}^*,\ph{12,4},\ph{8,3}^*,\ph{8,9}^*\}, \\
  & \{B_2[.1^2],B_2[.2],B_2[1^2.],B_2[2.]\}\\
 (1,12)& \{\ph{1,0},\ph{1,24},\ph{6,6}'',\ph{4,1},\ph{4,13}\},
  \{B_2[.1^2],B_2[1.1],B_2[2.]\} \\
 (3,4)& \{\ph{1,0},\ph{1,24},\ph{4,8},\ph{4,1},\ph{4,13},F_4^2[1]\},
  \{\ph{2,4}',\ph{2,16}',\ph{4,7}'\},\{\ph{2,4}'',\ph{2,16}''\ph{4,7}''\},
 \\
 (3,6)&  \{\ph{1,0},\ph{2,16}^*,\ph{8,3}^*,F_4[\tht^2]\},
\{\ph{1,24},\ph{2,4}^*,\ph{8,9}^*,F_4[\tht]\},
\\
 (3,12)& \{\ph{1,0},\ph{4,13},F_4[\tht]\},\{\ph{1,24},\ph{4,1},F_4[\tht^2]\}\\
 (4,8)&  \{\ph{1,0},\ph{9,10},F_4[i]\},\{\ph{1,24},\ph{9,2},F_4[-i]\}
\\
 (4,12)&  \{\ph{1,0},\ph{1,24},B_2[1.1]\},
\{\ph{4,1},B_2[.1^2],F_4[-i]\},
\{\ph{4,13},B_2[2.],F_4[i]\}\\
 (8,12)& \{\ph{1,0},F_4[-i]\},\{\ph{1,24},F_4[i]\}\\
\end{array}$$
Here, an entry $\ph{d,b}^*$ stands for the \emph{two} characters
$\ph{d,b}'$ and $\ph{d,b}''$. Again, we have omitted all singletons, as
well as the pairs $(1,8),(3,8)$ where there is just one non-singleton block.
\end{table}

\begin{proof}
The relative Weyl groups are described in \cite[Tab.~1 and~3]{BMM}, the
parameters for the corresponding Hecke algebras are predicted in
\cite[Tab.~8.1]{BM93}. The relevant block decompositions are taken from
\cite[Tab.~F.3]{GP} for $W(F_4)$ and from \cite[3.3.2, 3.3.3]{CM11} for the
primitive complex reflection groups $G_5$ and $G_8$ occurring as relative Weyl
groups for $d=3,6$ and $d=4$, respectively. 
\end{proof}

\subsection{Type $E_6$}   \label{subsec:E6}
For $G$ of type $E_6$, again most block sizes for the relevant algebras have
been determined in \cite[8.7]{TX23}. Here, we also compare the blocks to the
intersections of Harish-Chandra series. The only unipotent characters not
distinguished by their degree are the two cuspidal ones, so only for those is
there a choice in the HC-parametrisation.

\begin{prop}   \label{prop:E6}
 Conjecture~\ref{conj:TX} holds for type $E_6$ with respect to some complete
 HC-parametrisation.
\end{prop}

In this case and from now on, we do not display the block intersections
for reasons of space, but they can all be found in the github repository \cite{git}.

\begin{proof}
The relative Weyl groups are described in \cite[Tab.~1 and~3]{BMM}, the
parameters for the corresponding Hecke algebras are predicted in
\cite[Tab.~8.1]{BM93}. The relevant decomposition matrices are taken from
\cite[Tab.~F.4]{GP} for $W(E_6)$ and \cite[\S3.3]{CM11} for the 2-dimensional
groups. The decomposition matrices for $G_{25}$ were computed using the methods
expounded in Section~\ref{subsec:algo} (with just a
couple of cases taken from \cite{Pl14}).
\end{proof}

\subsection{Type $E_7$}   \label{subsec:E7}
The groups of type $E_7$ possess four pairs of unipotent characters not
distinguished by their degrees, so only for those is there a choice in the
HC-parametrisation.

\begin{prop}   \label{prop:E7}
 Conjecture~\ref{conj:TX} holds for type $E_7$ with respect to a suitable
 complete HC-parametrisation.
\end{prop}

\begin{proof}
The relative Weyl groups are described in \cite[Tab.~1 and~3]{BMM}, the
parameters for the corresponding Hecke algebras are predicted in
\cite[Tab.~8.1]{BM93}. The relevant blocks are taken from \cite[Tab.~F.5]{GP}
for $W(E_7)$ and \cite[\S3.3]{CM11} for 2-dimensional groups.
The decomposition matrices for $G_{26}$ were computed using the methods
explained in Section~\ref{subsec:algo}.
\end{proof}

Again, we do not print the actual intersections.
In Table~\ref{tab:bl E7} we collect the number of non-singleton blocks appearing
in the intersection of Harish-Chandra series for the relevant pairs $(d,e)$.

\begin{table}[htb]
\caption{Numbers of non-trivial block intersections for $E_7$}   \label{tab:bl E7}$\begin{array}{c|ccccccccccccc}
 e& d= 2& 3& 4& 5& 6& 7& 8& 9& 10& 12& 14& 18\\
\hline
 1& 4& 8& 4& 6& 6& 2& 8& 2& 5& 4& 2& 3\\
 3& & & 12& 6& 13& 4& 10& 4& 9& 6& 4& 4\\
 4& & & & 8& 12& 4& 16& 4& 8& 6& 4& 4\\
 5& & & & & 9& 4& 8& 4& 8& 4& 3& 3\\
 7& & & & & & & 4& 2& 3& 2& 2& 1\\
 8& & & & & & & & 4& 8& 6& 4& 4\\
 9& & & & & & & & & 3& 4& 1& 4\\
 12& & & & & & & & & & & 2& 4\\
\end{array}$
\end{table}

\subsection{Type $E_8$}   \label{subsec:E8}
For type $E_8$ we have not been able to determine the block decompositions
completely for all relevant cyclotomic Hecke algebras. More precisely, the
dimensions of the irreducible representations of the 4-dimensional reflection
group $G_{32}$ (and thus of its Hecke algebra), occurring as relative Weyl
group for the principal series when $d=3$ or $d=6$, are too large. The
situation for the 4-dimensional group $G_{31}$, occurring when $d=4$, is even
worse, as the data available in \Chevie\ are rather incomplete. In these two
situations we only obtained approximations to the actual block decompositions,
see the remarks at the end of Section~\ref{subsec:algo}. These approximations
turned out to be in accordance with the Trinh--Xue conjecture as well.

\begin{prop}   \label{prop:E8}
 Conjecture~\ref{conj:TX} holds for type $E_8$, except possibly when one of
 the two Harish-Chandra series is a principal $d$-series with $d\in\{3,4,6\}$.
\end{prop}

\begin{proof}
The relative Weyl groups are described in \cite[Tab.~1 and~3]{BMM}, the
parameters for the corresponding Hecke algebras are predicted in
\cite[Tab.~8.1]{BM93}. The relevant blocks are taken from \cite[Tab.~F.6]{GP}
for $W(E_8)$ and \cite[\S3.3]{CM11} for 2-dimensional groups.
\end{proof}

\subsection{Suzuki and Ree groups}

We now consider a closely related setting. In this section let $\bG$ be simple
of type $B_2,G_2$ or $F_4$ in characteristic~$p=2,3,2$ respectively, and assume
$F:\bG\to\bG$ is a Steinberg endomorphism which is not a Frobenius, so
$G:=\bG^F$ is a Suzuki group $\tw2B_2(q^2)$ or a Ree group $^2G_2(q^2)$ or
$\tw2F_4(q^2)$, for $q^2$ an odd power of $p$. Here, one can still define
$\Phi$-tori, which satisfy a Sylow theory, where now $\Phi$ is a cyclotomic
polynomial over $\QQ(\sqrt{p})$, see e.g.\ \cite[3.5.3]{GM20}. Furthermore, the
relative Weyl groups are again complex reflection groups, and there is a
$\Phi$-Harish-Chandra theory for unipotent characters as in the case of
Frobenius maps \cite[\S3A]{BMM}. It is therefore not surprising that the
following extension of Conjecture~\ref{conj:TX} holds:

\begin{prop}   \label{prop:SuzRee}
 The Suzuki and Ree groups satisfy the analogue of Conjecture~\ref{conj:TX} for
 all pairs of cyclotomic polynomials over $\QQ(\sqrt{p})$.
\end{prop}

\begin{proof}
The relative Weyl groups $W_\Phi$ of Sylow $\Phi$-tori in this case are cyclic
unless $G=\tw2F_4(q^2)$ and the roots of $\Phi$ have order dividing~8, see
e.g.\ \cite[Tab.~3.2]{GM20}. The parameters of the corresponding cyclotomic
Hecke algebras have been given in \cite{Lu76} for the Coxeter case and in
\cite[Tab.~8.1]{BM93} otherwise. The decomposition numbers for the non-cyclic
relative Weyl groups $G_8$ and $G_{12}$ are taken from \cite{CM11}.
\end{proof}

\section{Extension to the spetsial case}   \label{sec:spets}
We consider an extension of Conjecture~\ref{conj:TX} to spetsial complex
reflection groups. These include, in particular, all finite Coxeter groups.
We refer to \cite{BMM14} for notions and general facts about spetses.

\subsection{A generalised conjecture}
Let $W$ be a \emph{spetsial} (finite) complex reflection group on $V=\CC^n$
(see \cite[\S3]{MaICM}). Let moreover $\vhi\in\GL(V)$ be normalising $W$. Then
attached to the \emph{spets} $\GG:=(V,W\vhi)$ there is a set of \emph{unipotent
characters $\Uch(\GG)$} together with a \emph{degree} function
$$\Uch(\GG)\longrightarrow\CC[x],\quad \rho\mapsto \rho(1),$$
and an action of $N_{\GL(V)}(W\vhi)$ on $\Uch(\GG)$ (see \cite{Lu93} for the
non-rational Coxeter groups, \cite{Ma95} and \cite[Thm~8.4]{Ma00} for the
non-real cases, and also \cite{BMM14} for explicit tables for
primitive~$W$). If $W$ is a Weyl group, $\Uch(\GG)$ are just the (labels of the)
unipotent characters of an associated finite reductive group of rational
type $W\vhi$ together with their degree polynomials. For any root of unity
$\ze$ and $\Phi:=x-\ze$, there is defined a subset of
\emph{$\Phi$-cuspidal} elements of $\Uch(\GG)$. A (unipotent)
\emph{$\Phi$-cuspidal pair} of $\GG$ is a pair $(\LL,\la)$ where $\LL$ is a
$\Phi$-split Levi of $\GG$ and $\la\in\Uch(\LL)$ is $\Phi$-cuspidal.
For $(\LL,\la)$ a $\Phi$-cuspidal pair let $W_\GG(\LL,\la)$ be the stabiliser
of $\la$ in $W_\GG(\LL)$ (see \cite[Rem.~4.4]{BMM14}). It turns out by
inspection \cite[Rem.~4.15]{BMM14} that this is again a complex reflection
group. Attached to $(\LL,\la)$ is
also a Hecke algebra $\cH(W_\GG(\LL,\la),x)$ in one parameter $x$
(a specialisation of the generic Hecke algebra of $W_\GG(\LL,\la)$ for
parameters specified in \cite{Ma95,Ma97,BMM14}) naturally isomorphic (via the
specialisation $x\mapsto\ze$ and Tits' deformation theorem) to the group
algebra of $W_\GG(\LL,\la)$ over a splitting field. This Hecke algebra
should carry a canonical symmetrising form with associated Schur elements
denoted $S_\chi$, see Section~\ref{subsec:cyc}. According to
\cite[Ax.~4.31]{BMM14}, there is a partition
$$\Uch(\GG)=\displaystyle\coprod \Uch(\GG,(\LL,\la))$$
into \emph{$\Phi$-Harish-Chandra series}, where the union runs through
representatives of $W$-orbits of $\Phi$-cuspidal pairs $(\LL,\la)$ for $\GG$,
and as in the case of finite reductive groups (Theorem~\ref{thm:label}), for
any such $(\LL,\la)$ there is a bijection
$$\Psi_{\LL,\la}:\Irr(\cH(W_\GG(\LL,\la),x)) \rightarrow \Uch(\GG,(\LL,\la)),\qquad
   \vhi \mapsto \rho_\vhi,$$
such that the exact analogue of~$(*)$ holds.
It thus makes sense to propose the following, where $(\MM,\mu)$ is a
$\Phi'=x-\ze'$-cuspidal pair of $\GG$:

\begin{conj}   \label{conj:spets}
 In the above setting, the intersection of $\Phi$- resp.\ $\Phi'$-Hardish-Chandra
 series
 $$\Psi_{\LL,\la}(\Irr(\cH(W(\LL,\la),x)))
     \ \cap\ \Psi_{\MM,\mu}(\Irr(\cH(W(\MM,\mu),x)))$$
 admits a partition $\bigsqcup B_i$
 such that $\Psi_{\LL,\la}^{-1}(B_i)$ is a block of $\cH(W(\LL,\la),\ze')$
 and $\Psi_{\MM,\mu}^{-1}(B_i)$ is a block of $\cH(W(\MM,\mu),\ze)$ for
 all $i$, thus defining a bijection between those blocks of the
 two algebras appearing in this intersection.
\end{conj}

Note that for $W$ a Weyl group this specialises to Conjecture~\ref{conj:TX}.

It seems reasonable to expect that arguments similar to those used in
\cite[\S6]{TX23} might be applicable to show this for the infinite series of
imprimitive spetsial groups.

\begin{rem}   \label{rem:Enn}
\vbox{\ }
\begin{enumerate}[(a)]
\item Observe that the case of Conjecture~\ref{conj:spets} where both relative
 Weyl groups are cyclic again holds by the analogue of Corollary~\ref{cor:cyc}.
\item There is again an Ennola phenomenon: If the centre of $W$ (assumed to be
 irreducible) is cyclic of order $z$, then replacing $x$ by $\exp(2\pi i/z)x$
 induces a permutation on the set of unipotent characters as well as on the set
 of associated cyclotomic Hecke algebras (see \cite[4.2.6]{BMM14}) and hence of
 their block distributions, meaning that it suffices to check
 Conjecture~\ref{conj:spets} for just one representative in each orbit under
 this Ennola permutation (generalising Lemma~\ref{lem:Ennola}(a)).
\end{enumerate}
\end{rem}

In the sequel we verify Conjecture~\ref{conj:spets} for some exceptional cases.

\subsection{The non-crystallographic Coxeter groups}
The first interesting case for Conjecture~\ref{conj:spets} is of course that of
non-crystallographic Coxeter groups, that is, of types $I_2(m)$ with $m\ge3$,
$H_3$ and~$H_4$.

\begin{prop}   \label{prop:I2(m)}
 Conjecture~\ref{conj:spets} holds for type $I_2(m)$, $m\ge3$.
\end{prop}

\begin{proof}
The proper Levi subgroups of $I_2(m)$ are of maximal tori or of type $A_1$,
thus the only proper $d$-cuspidal pairs are of type $(\TT,1)$, for $\TT$ a
maximal torus containing a Sylow $d$-torus. The Sylow $d$-tori and hence their
relative Weyl groups are cyclic unless $d\in\{1,2\}$, where $d=2$ only occurs
when $m$ is even. Thus, by Corollary~\ref{cor:cyc} for checking the conjecture
we may assume that $d\le2$, and then by Ennola duality that $d=1$.   \par
First assume $m$ is odd and $d=1$. We need to consider $\Phi=x-\zeta$-series
for $\zeta$ some $m$th root of unity. Let $\ze_m=\exp(2\pi i/m)$. By the
action of Galois automorphisms we may assume $\ze=\zeta_m^{m/e}$ for some
divisor $e>2$ of $m$, a primitive $e$th root of unity. According to
\cite[Tab.~7.2]{GJ11} the Iwahori--Hecke algebra $\cH(I_2(m),(x,-1))$
specialised at $x=\ze$ has just one non-singleton block, containing the trivial
and the sign representation plus the representation denoted $\boldsymbol\si_e$.
On the other hand, the Hecke algebra for the principal $\Phi$-series is
$\cH(C_m,\bx)$ with
$$\bx=(1,\zeta_m^kx,x^2: |k|\le (m-1)/2,\ k\ne \pm m/e)$$
by \cite[Satz~6.10]{Ma95}. Its specialisation at $x=1$ again has just one
non-singleton block (see Section~\ref{subsec:cyclic}), containing the linear
characters corresponding to the parameters $1,x,x^2$. By the description in
loc.\ cit.\ these correspond to the unipotent principal series characters
labelled $1_W,\sgn,\boldsymbol\si_{m/e}$, so Conjecture~\ref{conj:spets} is
satisfied.
\par
When $m$ is even, the case when $d=1$ and $e\ge3$ is entirely similar to the
above, so it remains to deal with the case $d=1,e=2$.
Here, by \cite[Tab.~7.2]{GJ11} the Iwahori--Hecke algebra $\cH(I_2(m),(x,-1))$
specialised at $x=-1$ has just one non-singleton block, containing the four
linear characters. By Ennola duality again, the Hecke algebra
$\cH(I_2(m),(x,1))$ specialised at $x=1$ also has just one non-singleton block,
again consisting of the of the four linear characters. By the description in
\cite[Def.~6.3]{Ma95} the linear characters in both cases label the same four
unipotent characters, confirming  Conjecture~\ref{conj:spets} in this case.
In fact, this case also follows from the analogue of Lemma~\ref{lem:Ennola}(b).
\end{proof}

Note that this result generalises Proposition~\ref{prop:G2}, which is for
$I_2(6)$. There is also a twisted spets $\tw2I_2(m)$, as described in
\cite[\S6]{Ma95}. When $m$ is odd, this can be obtained as an Ennola twist
of the untwisted spets of $I_2(m)$, so Conjecture~\ref{conj:spets} holds by
Proposition~\ref{prop:I2(m)} with Remark~\ref{rem:Enn}(b). When $m$ is even,
this includes the case of the twisted finite reductive groups of types
$\tw2B_2$ and $^2G_2$. We have checked that Conjecture~\ref{conj:spets} holds
in this case as well.
\medskip

We next investigate $H_3$. In the following we write $e',e''$ for the two
factors of $\Phi_e$ over $\ZZ[\sqrt{5}]$, where $e\in\{5,10,15,20,30\}$. For
the parameters see \cite[Tab.~8.3]{BM93}. There are four pairs of unipotent
characters of $H_3$ having the same degree, so only for those is there a choice
in the HC-parametrisation.

\begin{prop}   \label{prop:H3}
 Conjecture~\ref{conj:spets} holds for type $H_3$ with respect to some complete
 HC-parametrisation. The block intersections are displayed in
 Table~\ref{tab:H3}.
\end{prop}

\begin{table}[htb]
\caption{Block intersections for $H_3$}   \label{tab:H3}
$\begin{array}{c|cc}
 (d,e)& \text{intersections}\\
\hline
 (1,3)& \{\ph{1,0},\ph{4,3},\ph{5,2}\},
  \{\ph{1,15},\ph{4,4},\ph{5,5}\}\\
 (1,5')& \{\ph{1,0},\ph{3,6},\ph{4,3}\},
  \{\ph{1,15},\ph{3,1},\ph{4,4}\}\\
 (1,10')& \{\ph{1,0},\ph{1,15},\ph{3,1},\ph{3,6}\},\\
  & \{H_2[\ze_5^3]\co1^2,H_2[\ze_5^3]\co2\},\{H_2[\ze_5^2]\co1^2,H_2[\ze_5^2]\co2\} \\
 (3,5')& \{\ph{1,0},\ph{4,3}\},\{\ph{1,15},\ph{4,4}\} \\
 (3,6)&  \{\ph{1,0},\ph{5,5}\},\{\ph{1,15},\ph{5,2}\}\\
 (5',5'')& \{\ph{1,0},H_2[\ze_5^3]\co1^2\},\{\ph{1,15},H_2[\ze_5^2]\co2\},
\{\ph{4, 3},H_2[\ze_5^2]\co1^2\},\{\ph{4, 4},H_2[\ze_5^3]\co2\}\\
 (5',10')& \{\ph{1,0},H_2[\ze_5^2]\co2\},\{\ph{1,15},H_2[\ze_5^3]\co1^2\},
   \{\ph{3,1},H_2[\ze_5^2]\co1^2\},\{\ph{3,6},H_2[\ze_5^3]\co2\}\\
 (5',10'')&  \{\ph{1,0},H_2[\ze_5^2]\co1^2\},\{\ph{1,15},H_2[\ze_5^3]\co2\},
   \{H_2[\ze_5^3]\co1^2,H_2[\ze_5^2]\co2\}\\
\end{array}$
\medskip

Again, we have omitted all singletons, as well as the pairs\\
$(1,2),(1,6),(3,10')$ where there is just one non-singleton block.
\end{table}

\begin{proof}
The relative Weyl groups and the parameters for the corresponding cyclotomic
Hecke algebras are described in \cite[Tab.~8.3]{BM93}. The decomposition
matrices for $W(H_3)$ were computed by M\"uller \cite[\S4]{Mu97}.
\end{proof}

We now turn to $H_4$. There are 29 pairs of characters with equal degree, plus
one family of four characters with that property. So for these is there a choice
in the HC-parametrisation.

\begin{prop}   \label{prop:H4}
 Conjecture~\ref{conj:spets} holds for type $H_4$ with respect to a suitable
 complete HC-parametrisation.
\end{prop}

We refrain again from displaying the block intersections.

\begin{proof}
The relative Weyl groups and the parameters for the corresponding cyclotomic
Hecke algebras are described in \cite[Tab.~8.3]{BM93}. The decomposition
matrices for $W(H_3)$ were computed by M\"uller \cite[\S4]{Mu97}. The
remaining block decompositions were obtained as described in
Section~\ref{subsec:algo}.
\end{proof}

\subsection{Verifying the spetsial conjecture}
Finally, we consider Conjecture~\ref{conj:spets} for some of the primitive
spetsial reflection groups $W$. 

The degree polynomials of the unipotent characters of the spets associated
to the complex reflection group $G_4$ are all distinct, so there is a unique
complete HC-parametrisation.

\begin{prop}   \label{prop:G4}
 Conjecture~\ref{conj:spets} holds for types $G_4, G_6, G_8$ and $G_{24}$. The block
 intersections for $G_4$ are displayed in Table~\ref{tab:G4}.
\end{prop}

\begin{table}[htb]
\caption{Block intersections for $G_4$}   \label{tab:G4}
$\begin{array}{c|cc}
 (d,e)& \text{intersections}\\
\hline
 (1,2)& \{[\ph{1,0},\ph{1,4},\ph{1,8},\ph{2,3},\ph{2,1},\ph{3,2}\},
  \{Z_3[2],Z_3[1^2]\}\\
 (1,3')& \{\ph{1,0},\ph{1,8}\},\{\ph{2,5},\ph{2,1}\}\\
 (1,4)& \{\ph{1,0},\ph{2,5},\ph{3,2}\}\\
 (1,6')& \{\ph{1,0},\ph{1,4},\ph{2,1}\},\{Z_3[2],Z_3[1^2]\}\\
 (3',3'')& \{\ph{1,0},Z_3[1^2]\},\{\ph{2,5},Z_3[2]\}\\
 (3',4)& \{\ph{1,0},\ph{2,5}\}\\
 (3',6')& \{\ph{1,0},Z_3[2]\},\{\ph{2,1},Z_3[1^2]\}\\
 (3',6'')& \{\ph{1,0},Z_3[1^2]\},\{\ph{1,8},Z_3[2]\}\\
\end{array}$

Again, we have omitted all singletons, as well as pairs that can be obtained\\
by a Galois automorphism from pairs displayed in the table, like $(1,3''),(1,6'')$.
\end{table}

We also have a partial result for the 5-dimensional group $G_{33}$. Here,
again, the dimensions of some irreducible representations are too large to be
handled by our methods, and we only obtained approximate block decompositions.

\begin{prop}   \label{prop:G33}
 Conjecture~\ref{conj:spets} holds for type $G_{33}$ except possibly when one
 of the two Harish-Chandra series is the principal $d$-series, with
 $d\in\{1,2\}$.
\end{prop}

In this case, the 3-dimensional group $G_{26}$ occurs as a relative Weyl group
for $d=3,6$.


\end{document}